\documentclass[letterpaper,11pt,reqno]{amsart} 
\usepackage[portrait,margin=1in]{geometry} 
\usepackage{amssymb}
\usepackage{graphicx}
\usepackage{amsmath}
\usepackage{tikz}
\usepackage{hyperref}
\usepackage{dsfont}
\usepackage{lmodern}
\usepackage{mathabx}
\usepackage{comment}
\usepackage[spacing=true,kerning=true,babel=true,tracking=true]{microtype}
\usepackage[shortcuts]{extdash}

\usetikzlibrary{graphs, calc}

\hypersetup{
	pdfstartview={XYZ null null 1.00}, 
	pdfpagemode=UseNone, 
	colorlinks,
	breaklinks, 
	linkcolor=blue,
	urlcolor=blue, 
	anchorcolor=blue,
	citecolor=blue
}

\newtheoremstyle{bfnote}%
{}{}%
{\slshape}{}%
{\bfseries}{\bfseries.}%
{ }%
{\thmname{#1}\thmnumber{ #2}\thmnote{ \ep{\normalfont{}#3}}}

\theoremstyle{bfnote}
\newtheorem{theorem}{Theorem}[section]

\newtheorem{lemma}[theorem]{Lemma}
\newtheorem{prop}[theorem]{Proposition}
\newtheorem{cor}[theorem]{Corollary}

\newtheorem{claim}[theorem]{Claim}

\theoremstyle{definition}
\newtheorem{example}[theorem]{Example}

\newtheorem{definition}[theorem]{Definition}
\newtheorem{remark}[theorem]{Remark}

\DeclareMathOperator{\ext}{ext}

\newcommand{\R}{\mathbb{R}}

\renewcommand{\epsilon}{\varepsilon}
\newcommand{\eps}{\epsilon}
\renewcommand{\phi}{\varphi}
\renewcommand{\theta}{\vartheta}
\renewcommand{\leq}{\leqslant}
\renewcommand{\geq}{\geqslant}

\newcommand{\bemph}[1]{{\normalfont#1}}
\newcommand{\ep}[1]{\bemph{(}#1\bemph{)}}

\newcommand{\red}[1]{{\color{red}#1}}

\newcommand{\todo}[1]{\red{\textbf{\upshape{[#1]}}}}

\newcommand{\Aut}{\mathrm{Aut}}

\newcommand{\w}{\mathbf{w}}
\newcommand{\abs}[1]{\left\vert#1\right\vert}
\title{\sffamily Measurable one-ended spanning trees}
\date{}
\author[Bowen]{Matt Bowen}
\address{\normalfont (MB) 
Mathematical Institute, University of Oxford, Oxford, UK}
\email{matthew.bowen@maths.ox.ac.uk}

\author[Girao]{António Girão}
\address{\normalfont (AG) 
Mathematical Institute, University of Oxford, Oxford, UK}
\email{antonio.girao@maths.ox.ac.uk}

\author[Sanchez]{H\'ector Jard\'on-S\'anchez}
\address{\normalfont (HJS) 
Institute of Mathematics, Jagiellonian University, Krak\'ow, Poland}
\email{hector.jardon.sanchez@uj.edu.pl}

\author[Terlov]{Grigory Terlov}
\address{\normalfont (GT) 
Department of Statistics and Operations Research,
University of North Carolina at Chapel Hill, NC, USA}
\email{gterlov@unc.edu}

\thanks{}

\numberwithin{equation}{section}

\usepackage{titlesec}
\titleformat{\section}[block]{\large\bfseries\sffamily}{\thesection.}{1ex}{}
\titleformat{\subsection}[block]{\bfseries\sffamily}{\thesubsection.}{1ex}{}
\titleformat{\subsubsection}[block]{\itshape}{\bfseries\upshape\sffamily\thesubsubsection.}{1ex}{}
\titlespacing*{\section}{0pt}{*3}{*1}
\titlespacing*{\subsection}{0pt}{*3}{*1}
\titlespacing*{\subsubsection}{0pt}{*3}{*1}

\begin{document}

\begin{abstract}
    We show that a one-ended, locally finite, measurable graph on a standard probability space admits a measurable one-ended spanning subtree if and only if it is measure-hyperfinite. This answers a question posed by Bowen, Poulin, and Zomback and extends recent results of Tim\'ar and Conley, Gaboriau, Marks, and Tucker-Drob.
\end{abstract}

\maketitle
%\tableofcontents
%%%%%%%%%%%%%%%%%%%%%%%
\section{Introduction and main results}
%%%%%%%%%%%%%%%%%%%%%%%
In this paper, we are interested in measurable combinatorics, the study of properties and applications of graphs defined on measure spaces. Our main objects of interest are measurable graphs. A Borel graph $\mathcal{G}=(X,E)$ is a graph whose vertex set $X$ is a standard Borel space and the edge set $E$ is a Borel subset of $X^2$. Such a graph is called measurable if $X$ is  equipped with a Borel probability measure $\nu$ (e.g.\ $[0,1]$ with the Lebesgue measure). Background on this subject can be found in \cite{kechris.marks,pikhurko2020borel}. One of the central structural notions in this field is hyperfiniteness: a measurable graph $\mathcal{G}$ on $(X,\nu)$ is called (measure-)\textbf{hyperfinite} (or $\nu$-hyperfinite for short) if it can be written as an increasing union of measurable graphs whose connected components are finite $\nu$-a.e.

While the compactness theorem from first-order logic implies that many classical results from finite graph theory immediately generalize to infinite graphs, these extensions rely on the axiom of choice and typically only yield non-measurable results. 
In contrast, if one is only interested in measurable or Borel solutions to combinatorial problems, then the situation is quite different, even in the very constrained setting of acyclic graphs.  
For example, Marks \cite{Marks.determinacy} has constructed, for all $d\geq 2$, a $d$-regular acyclic Borel graph with no Borel $d$-coloring, showing that the Borel version of Brooks' theorem need not hold. 
This was later extended to give examples even for hyperfinite graphs \cite{conley2020borel}.
Similarly, Kun has constructed for all $d\geq 2$ a $d$-regular acyclic measurable graph with no measurable perfect matching, showing that Hall's theorem need not hold in the measure setting \cite{gabor.new}.

Although these examples are acyclic, they are leafless and hence belong to a quite special subclass of trees. The following definition constitutes a principal obstruction to well-behaved measurable combinatorics.
\begin{definition}
    We say that a connected locally finite graph $G$ is $\mathbf{\kappa}$\textbf{-ended} if $\kappa$ is the supremum number of infinite connected components that can be obtained by removing a finite set of vertices from $G.$  We say that a locally finite measurable graph $\mathcal{G}$ on $(X,\nu)$ is $\kappa$\textbf{-ended} if for $\nu$-a.e.\ $x\in X$ its $\mathcal{G}$-connected components are $\kappa$-ended, and we say that $\mathcal{G}$ admits a measurable $\kappa$\textbf{-ended spanning tree} if there is a measurable subgraph $\mathcal{T}\subseteq \mathcal{G}$ which is acyclic, one-ended, and has the same connected components as $\mathcal{G}$ $\nu$-a.e.
\end{definition}

Notice that $d$-regular trees for $d\geq 2$ (the aforementioned examples, in particular) must have at least two ends.  On the other hand, a locally finite tree is one-ended if and only if its vertex set can be exhausted by iteratively removing leaves. This property enables measurable variants of many algorithms from finite combinatorics that do not hold for arbitrary measurable graphs. Indeed, while the above examples show that the measurable versions of Brooks' theorem and Hall's theorem fail in general, they are known to hold for measurable graphs that admit measurable one-ended spanning trees, as was shown by Conley, Marks, and Tucker-Drob \cite{brooks} and Bowen, Kun, and Sabok \cite{bowen.kun.sabok}, respectively.  In addition, measurable one-ended spanning trees have found several recent applications, including measurable solutions to Tarski’s circle squaring problem and the construction of measurable $k$-factors and integral flows \cite{bowen.kun.sabok,bowen2025uniform}, Borel balanced orientations \cite{bowen2022one}, and factors of Euclidean spaces and Poisson point processes \cite{timar2004tree,timar2021nonamenable,Timar_opt_tail}.

With these applications in mind, it is natural to ask when a locally finite measurable graph $\mathcal{G}$ admits a measurable one-ended spanning tree. Classically, Halin’s theorem \cite{halin1964unendliche} implies that an infinite locally finite graph admits a one-ended spanning tree if and only if the graph itself is one-ended.
In the measurable setting, the situation is more subtle. The existence of a measurable one-ended spanning tree of a measurable graph $\mathcal{G}$ on $(X,\nu)$ implies that $\mathcal{G}$ is ($\nu$-)hyperfinite, so at the very least hyperfiniteness is a necessary condition. If in addition $\mathcal{G}$ is probability measure preserving (pmp), a result of Adams \cite{adams} shows that any measurable spanning tree of a hyperfinite graph has at most two ends. This was later sharpened by Tim\'ar \cite{Timar_oneended} and by Conley, Gaboriau, Marks, and Tucker-Drob \cite{cgmtd}, who proved that under the same hypotheses, the tree can in fact be chosen to be one-ended whenever $\mathcal{G}$ is one-ended.

Outside the pmp setting, hyperfinite graphs need not have at most two ends anymore. Indeed, there exist hyperfinite Borel graphs with infinitely many ends (e.g.\ see \cite[Example 1.2]{CTT22}). 
We note that by \cite[Proposition 2.1]{Miller:thesis} any locally countable measurable graph on $(X,\nu)$ is measure class preserving (mcp) on a $\nu$-conull set. 
A natural extension of Adams’ result for such graphs was recently obtained in \cite{AnushRobin}, where the classical notion of ends is replaced by a weighted analogue, known as nonvanishing ends. 
To our knowledge, the only result for spanning trees in the mcp generality stated directly in terms of classical ends is due to Bowen, Poulin, and Zomback \cite{bowen2022one}, who showed that Schreier graphs of free Borel actions of one-ended amenable groups admit measurable one-ended spanning trees and asked if it is possible to do this in general for locally finite hyperfinite mcp graphs. 

Our main result shows that hyperfiniteness is, in fact, the only obstruction to finding measurable one-ended spanning trees, thereby giving a positive answer to \cite[Question 1]{bowen2022one}.

\begin{theorem}\label{measure_main}
    Let $\mathcal{G}$ be a locally finite, one-ended, mcp graph on a standard probability space $(X,\nu)$. Then $\mathcal{G}$ is hyperfinite if and only if it admits a measurable one-ended spanning tree a.e.
\end{theorem}

We highlight that Theorem~\ref{measure_main} enables many extensions of results that rely on \cite{Timar_oneended,cgmtd} to the mcp setting. 
Similarly to the recent rise of activity in the mcp setting in measured combinatorics and measured group theory \cite{Ts:hyperfinite_ergodic_subgraph,AnushRobin, bowen2022one, CTT22, poulin-anticost, OscillatingEnd}, there has been an increasing amount of interest in percolation beyond the unimodular setting \cite{Hutchcroft20,CTT22,HutchcroftPan24dim,wamen,HeavyClusterRepulsion}. In the latter setting, our main result gives the following generalization of \cite[Corollary 1.2]{Timar_oneended}.

\begin{cor}\label{perc_mian}
    For a connected, locally finite, one-ended graph $G$, let $\Gamma$ be a closed subgroup of $\Aut(G)$ that acts quasi-transitively on $G$. Then $\Gamma$ is amenable if and only if there is a $\Gamma$-invariant random spanning tree of $G$ with one end a.s. 
    Moreover, such a spanning tree can be constructed as a factor of i.i.d.
\end{cor}

We note we phrased Corollary~\ref{perc_mian} in terms of amenability of the group, while \cite[Corollary 1.2]{Timar_oneended} is stated in terms of amenability of the graph (i.e.\ the Cheeger constant is zero). For unimodular quasi-transitive graphs these notions of amenability are equivalent by the Soardi--Woess--Salvatori theorem \cite{Soardi-Woess,Salvatori}. One can also replace the assumption of amenability on $\Gamma$ by an equivalent hypothesis of weighted-amenability of $G$ with respect to $\Gamma$, which was introduced in \cite{BLPS99inv} and further considered in \cite{wamen}.

%%%%%%%%%%%%%%%%%%%%%%%
\subsection*{Organization}
%%%%%%%%%%%%%%%%%%%%%%%
In Section~\ref{basic_def} we recall a few basic definitions in both finite graph theory and measured combinatorics. Section~\ref{sec:proof} is dedicated to the proof of our main result, Theorem~\ref{measure_main}. It starts with an outline of the proof and is further divided in  Sections~\ref{sec:substantial}, ~\ref{sec:out of tree}, ~\ref{finite}, and ~\ref{infinite} covering the propositions that the proof relies on. Finally, Section~\ref{sec:perc} recalls relevant background in percolation theory and presents the proof of Corollary~\ref{perc_mian}.

%%%%%%%%%%%%%%%%%%%%%%%
\section{Basic definitions}\label{basic_def}
%%%%%%%%%%%%%%%%%%%%%%%

%First, we recall basic definitions. For a connected locally finite graph $G$, we say that a set of vertices $A$ is \textbf{end-convergent} in $G$ if for any finite subgraph $F\subset G$ all but finitely many elements of $A$ are contained in the same connected component of $G\setminus F$. 
%Two end-convergent sets $A$ and $B$ are said to be equivalent if $A\cup B$ is also end-convergent. An \textbf{end} of $G$ is an equivalence class of end-convergent sets. We say that a set of vertices \textbf{converges to an end} $\xi$ if it belongs to the equivalence class $\xi$. {\todo{matt: do we need this definition?}}

%\todo{Greg: Place here whatever you need to define we will organize it later.}

%\todo{Left to add: bracket for induced subgraph. \\ add transversal and maybe something about different notions of boundary, at least outer boundary, edge boundary, and visible boundary.}

\subsection{Graph theory}

Let $G = (V(G),E(G))$ be a graph. If $U \subset V(G)$, we denote the \textbf{induced subgraph} of $G$ by $U$ as $G[U]$. Recall that this graph has vertex set $U$ and edge set $E(G[U]) = E(G) \cap (U\times U)$. The \textbf{outer vertex boundary} of such a $U\subset V(G)$ is the set $$
\partial_v^{\mathrm{out}} (U) := \{v \in V(G) \backslash U : \exists u \in U \ \text{such that}\ (u,v) \in E(G)\}.
$$
If $u,v\in V(G)$, we denote by $d_G (u,v)$ the graph distance between $u$ and $v$, i.e., the length of a shortest path with end-vertices $u$ and $v$.

We denote the degree of a vertex $v \in V(G)$ by $\deg_G (v)$. The \textbf{cycle space of $G$}, denoted as $\mathcal C (G)$, is the vector subspace of $\mathbb  Z_2 [E(G)]$ generated by the indicator functions of cycles. A \textbf{fundamental cycle in $G$ (with respect to $T$)} is a cycle in $G$ that contains exactly one edge in $ E(G) \backslash E(T)$. Note that for any spanning tree $T$ of $G$, the family of fundamental cycles with respect to $T$ is a basis of the cycle space $\mathcal C (G)$.

Recall that a finite graph is said to be \textbf{$k$-edge-connected} if the removal of any $k$ edges does not disconnect it. Menger's theorem \cite{Menger1927} states that a graph is $k$-edge-connected if and only if every pair of vertices admits $k$-many pairwise edge-disjoint paths between them.
The following theorem is a consequence of the Nash–Williams tree packing theorem \cite{nash1961edge,Tutte}.

\begin{theorem}[{\cite[Corollary 2.4.2]{diestel}}]\label{nash-williams}
    Every $2k$-edge-connected multi-graph admits $k$ edge-disjoint spanning trees.
\end{theorem}

\subsection{Borel and measured combinatorics}\label{sec:prelim_measured}

Let $X$ denote a standard Borel space. A \textbf{countable Borel equivalence relation (cber)} on $X$ is a Borel subset $\mathcal{R} \subseteq X^2$ whose elements define an equivalence relation with countable equivalence classes on $X$. 
Similarly, a \textbf{finite Borel equivalence relation (fber)} is a cber whose components are finite. 
A \textbf{Borel graph}  on $X$ is a Borel subset $\mathcal{G} \subseteq X^2$ whose elements define the edges of a graph on the vertex set $X$.  A Borel graph $\mathcal{G}$ is \textbf{locally countable} (resp.\ \textbf{locally finite}) if every $x \in X$ has countably (resp.\ finitely) many $\mathcal{G}$-neighbors. 
If $\mathcal{G}$ is a locally countable Borel graph we use $\mathcal{R}_{\mathcal{G}}$ to denote its connectedness relation, that is the cber on $X$ defined by belonging to the same $\mathcal{G}$-connected component. 
For a cber $\mathcal{R}$, a Borel graph $\mathcal{G}$, and $x\in X$, we denote by $[x]_{\mathcal{R}}$ (resp.\ $[x]_{\mathcal{G}}$) the $\mathcal{R}$-equivalence class (resp.\ $\mathcal{G}$-connected component) that contains $x$. 
Given a Borel $A\subseteq X$ and a Borel graph $\mathcal{G}$ we denote the induced subgraph of $\mathcal{G}$ on $A$ by $\mathcal{G}[A]$.

A cber $\mathcal{R}$ on a standard probability space $(X,\nu)$ is said to be $\nu$\textbf{-probability measure preserving (pmp)} (resp.\ $\nu$\textbf{-measure class preserving (mcp)}) if every Borel $\mathcal{R}$-invariant automorphism $\gamma$ of $X$ preserves $\nu$ (resp.\ $\nu$-null sets).
A measurable graph $\mathcal{G}$ is said to be pmp (resp.\ mcp) whenever its connectedness relation $\mathcal{R}_{\mathcal{G}}$ is. The lack of invariance in the measure class preserving cber is quantified by a relative weight function (that is a cocycle) on $X$. First recall that a Borel map $\w \colon \mathcal{R} \to \R^+$ is a \textbf{Borel cocycle} on $\mathcal{R}$ if for every three $\mathcal R$-related $x,y,z \in X$ the following identity holds: 
$$
\w^x (y)\w^y (z)  = \w^x (z).
$$
Now, if $\mathcal{R}$ is an mcp cber on $(X,\nu)$ then there exists a Borel cocycle $\w_\nu \colon \mathcal{R}_{\mathcal{G}} \to \mathbb R^+$ satisfying the \textbf{(weighted) Mass Transport Principle (MTP)}: for every $f \colon \mathcal{R}_{\mathcal{G}}\to \mathbb R_{\geq 0}$, 
\begin{equation}\label{mtp}
\int_X \sum_{y\in[x]_{\mathcal{R}}} f(x,y) d\nu (x) = \int_X \sum_{x\in[y]_{\mathcal{R}}} f(x,y) \w_\nu^y (x) d\nu (y).
\end{equation}
In this context, the cocycle value $\w_\nu^y(x)$ can be interpreted as ``the weight of $x$ relative to $y$'' and it is referred as the \textbf{Radon--Nikodym cocycle} on \ $\mathcal{R}_{\mathcal{G}}$ (resp.\ $\mathcal{G}$) with respect to $\nu$. As mentioned above, by \cite[Proposition 2.1]{Miller:thesis}, every cber on $(X,\nu)$ is mcp on a $\nu$-conull set. Finally, note that a cber is pmp whenever $\w_\nu\equiv 1$. To simplify the notation we will omit the subscript $\nu$ whenever it is clear from the context.

If $\mathcal{G}$ is an mcp graph on $(X,\nu)$, the vertex measure $\nu$ extends to an \textbf{edge measure} on $X^2$, denoted by $\mu$ throughout this paper, which is defined by 
$$
\mu (A)=\mu_{\nu} (A) := \int_X \deg_{\mathcal{G}\cap A}(x) d\nu (x),
$$
for each Borel subset $A \subseteq X^2$ and where $\deg_{\mathcal{G}\cap A}(x)$ denotes the degree of $x$ in $\mathcal{G}\cap A$. In particular, the measured graph $\mathcal{G}$ is locally finite whenever $\deg_{\mathcal{G}} (x)<\infty$ for $\nu$-a.e.\ Under the assumption of local finiteness there exists an equivalent measure\footnote{Precisely, define $\nu' (A) := \int_A \deg_{\mathcal G} (x)^{-1}d\nu(x)$ for every Borel $A\subset X$.} $\nu' \sim \nu$ whose edge measure $\mu'$ is such that $\mu' (\mathcal G) < \infty$. Without loss of generality, we will assume throughout the paper that $\mu (\mathcal G)< \infty$.

 Note that by the MTP, as in \eqref{mtp}, we can rewrite $\mu(\mathcal{G})$ as follows
\begin{align}\label{weighted_deg}
        \mu (\mathcal{G}) &= \int_X \deg_{\mathcal{G}}(x) d\nu (x)= \int_X \sum_{y\in[x]_{\mathcal{G}}} \mathds{1}_{(x,y)\in \mathcal{G}} d\nu (x)\notag\\
                  \textnormal{[by MTP]}\quad &=\int_X \sum_{x\in[y]_{\mathcal{G}}}\mathds{1}_{(x,y)\in \mathcal{G}}\w^y(x) d\nu (y)=\int_X \sum_{x: (x,y)\in\mathcal{G}}\w^y(x) d\nu (y).
\end{align}

\section{Proof of Theorem~\ref{measure_main}}\label{sec:proof}
%%%%%%%%%%%%%%%%%%%%%%%

Throughout this section $\mathcal{G}$ denotes a locally finite mcp graph on a standard probability space $(X,\nu)$ unless otherwise stated. As pointed out in the preliminaries, we may assume without loss of generality that $\mu (\mathcal G) < \infty$, where $\mu$ denotes the edge measure associated to $\nu$.

We proceed to sketch an outline of the proof. Our basic strategy towards constructing a measurable one-ended spanning tree is obtaining a nested sequence of \emph{substantial} measurable subgraphs $(\mathcal G_n)_{n\in \mathbb N}$ of $\mathcal{G}$ such that $\mu(\mathcal{G}_n)\to0$. For a precise statement of this approach, see Proposition~\ref{basic} below the following definition of substantial.

\begin{definition}
    Let $\mathcal{G}$ be an mcp graph on a standard probability space $(X,\nu)$. A measurable subgraph $\mathcal{G}'\subseteq \mathcal{G}$ is \textbf{substantial} if a.e.\ $\mathcal{G}'$-component is either one-ended or is an isolated vertex and a.e.\ one-ended $\mathcal{G}$-component contains exactly one one-ended $\mathcal{G}'$-component. Given such as subgraph, we let $O (\mathcal{G}')$ denote the measurable set of vertices in one-ended components and $I(\mathcal{G}')$ be the union of isolated vertices.
\end{definition}

\begin{example}\label{ex:substantial}
    Let $\mathcal T$ be an acyclic one-ended Borel graph on a standard Borel space $X$. For each $n\in \mathbb N$, the set $L^{>n}\subset X$ consisting of those $x\in X$ at distance at least $n$ from a leaf, i.e.\ a degree one vertex, is Borel. The induced subgraph $\mathcal T [L^{>n}]$ is substantial. Each $\mathcal T$-connected component consists of a unique $\mathcal T [L^{>n}]$-connected component together with a union of finite trees, each attached by a single edge.
\end{example}

The above example motivates the following proposition.

\begin{prop}\label{basic}
    Let $\mathcal{G}$ be a one-ended, locally finite, hyperfinite, mcp graph with $\mu(\mathcal{G})$ finite. Then $\mathcal{G}$ has a measurable one-ended spanning tree if and only if 
    there exists a nested sequence of substantial measurable subgraphs $(\mathcal G_n)_{n\in \mathbb N}$ of $\mathcal{G}$ such that $\mu(\mathcal{G}_n)\to0$ as $n\to\infty$.
\end{prop}

The proof of Proposition \ref{basic} can be found in Section \ref{sec:substantial}. In order to prove Theorem \ref{measure_main}, we are only left with the task of finding some $\eps$ such that, whenever we are given a substantial subgraph $\mathcal G' \subset \mathcal G$ we can find a further substantial subgraph $\mathcal G'' \subset \mathcal G'$ such that $\mu (\mathcal G'') \leq (1 - \eps) \mu (\mathcal G')$. Roughly speaking, we must ensure that we can always delete ``sufficiently many edges'' from a locally finite one-ended hyperfinite mcp graph while preserving the end structure. %We prove that, under three different regimes, we are able to delete sufficiently many edges. The analysis of is treated independently in Sections~\ref{sec:out of tree}, ~\ref{finite}, and ~\ref{infinite}.

By hyperfiniteness, the Borel graph $\mathcal G$ admits a measurable spanning tree $\mathcal T \subset \mathcal G$. Our first edge-deletion result is the following and is proven in Section \ref{sec:out of tree}. 

\begin{prop}\label{tree measure}
    Let $\mathcal G$ be a one-ended, locally finite, hyperfinite, mcp graph with measurable spanning tree $\mathcal T$. Then, for every $\eps > 0$, there exists $\mathcal G' \subseteq \mathcal G$ one-ended with $\mathcal R_{\mathcal G'} = \mathcal R_{\mathcal G}$, such that $\mathcal T \subset \mathcal G'$, and $
    \mu( \mathcal G') \leq \mu (\mathcal T) + \eps$.
\end{prop}

The above proposition implies that, without loss of generality, we may assume that a measurable spanning tree of $\mathcal G$ constitutes most of the edge-mass of the graph. Therefore, we focus our attention now in deleting edges from measurable spanning trees $\mathcal T \subset \mathcal G$. 

The analysis of this scenario depends on whether or not edges belong to the core of $\mathcal T$, defined as follows. 

\begin{definition}\label{def:core}
 Let $\mathcal G$ be a Borel graph and $\mathcal T$ a Borel spanning tree. Then, the \textbf{core} of $\mathcal T$ (relative to $\mathcal G$) is defined as the Borel subset
    $$
    \mathrm{Core}_\mathcal{T}(\mathcal{G})=\{e\in \mathcal{T}: e \textnormal{ is infinitely many fundamental cycles relative to } \mathcal{T}\}.
    $$
\end{definition}

In Proposition \ref{fundamental}, we observe that the tree $\mathcal T$ is one-ended precisely when the core is empty. Perhaps it is then counterintuitive at first that our next step is to ensure that the whole tree belongs to the core. As we will discuss in the next paragraphs, this ensures strong connectivity properties which allow us to progress. In Section~\ref{finite} we prove the following. 

\begin{prop}\label{boundary connectivity}
    Let $\mathcal{G}$ be a one-ended, locally finite, hyperfinite, mcp graph with measurable spanning tree $\mathcal{T}$ such that $\mu (\mathcal T) < \infty$. Then, for every $\varepsilon>0$, there is a Borel subset $\mathcal{S}\subseteq \mathcal{T}\setminus \mathrm{Core}_\mathcal{T} (\mathcal{G})$ with $\mu(\mathcal S)>(1-\varepsilon)\mu(\mathcal{T}\setminus \mathrm{Core}_\mathcal{T} (\mathcal{G}))$ such that $\mathcal{G}\setminus \mathcal{S}$ is substantial.
\end{prop} 

We can then assume moreover that most of the mass of a measurable spanning tree belongs to the core. In Lemma~\ref{nocore} we show that $\mathcal{T}\setminus \mathrm{Core}_\mathcal{T(G)}$ induces a subgraph of $\mathcal T$ with finite connected components. Therefore, upon contracting these, we may assume, up to a small error, that the whole tree belongs to the core. This ensures sufficient connectivity to apply the tree packing theorem of Nash-Williams, as stated in Theorem~\ref{nash-williams}. In Section~\ref{infinite} we will prove the following. 

\begin{prop}\label{sparse_substential}
    Let $\mathcal{G}$ be a locally finite, hyperfinite, one-ended mcp graph on a finite measure space $(X,\nu)$ with $\mu(\mathcal{G})<\infty$. Assume $\mathcal{T}\subseteq\mathcal{G}$ is a measurable spanning tree of $\mathcal{G}$ such that $\mathrm{Core}_\mathcal{T}(G)=\mathcal{T}$. 
    Then, for every $\eps >0$, there is a substantial Borel subgraph $\mathcal{G}'\subseteq \mathcal{G}$ with $\mu(\mathcal{G}')< \mu(\mathcal{G})-\mu(\mathcal{T})/3 + \eps.$ 
\end{prop}

While the following subsections are devoted to proving the above statements, we now present the proof of our main result from them. 

\begin{proof}[Proof of Theorem \ref{measure_main}]
    %By Proposition \ref{basic.sketch} it suffices to show that for every  locally finite hyperfinite one-ended Borel graph $\mathcal G$ with $\mu (\mathcal G) < \infty$ admits a substantial Borel subgraph $\mathcal G' \subset \mathcal G$ with $\mu (\mathcal G') \leq C \mu (\mathcal G)$, for some $C\in (0,1)$. 

    Let $\mathcal T$ be a measurable spanning tree for $\mathcal G$. By Proposition~\ref{basic} it is enough to construct a nested sequence of substantial subgraphs with vanishing edge-measure. We do this showing that there exists $0 < C<1$ such that every locally finite, hyperfinite, one-ended mcp graph $\mathcal G$ with $\mu (\mathcal G) < \infty$ admits a substantial subgraph $\mathcal G'\subset \mathcal G$ with $\mu (\mathcal G') \leq C \mu (\mathcal G)$.

    Suppose $\eps >0$ to be chosen later. 
    If $\mu (\mathcal T) < (1 - \eps) \mu(\mathcal G)$, then by Proposition \ref{tree measure} there exists $\mathcal G'$ as desired. Therefore, we may assume that $\mu (\mathcal T) \geq (1 - \eps) \mu(\mathcal G)$.

    Again, we may assume that $\mu (\mathrm{Core}_\mathcal{T}(G))\geq (1 - \eps) \mu (\mathcal{T})$, or otherwise the proof would be concluded by Proposition \ref{boundary connectivity}. Contracting the edges in the finite connected components of the core -- this is at most $\eps \mu (\mathcal T)$ of the mass -- we may assume that $\mathrm{Core}_\mathcal{T}(G)=\mathcal{T}$. Indeed, Lemma \ref{nocore} shows that all connected components of the core are finite, unles they already are the whole tree. Finally, in this case, by Proposition \ref{sparse_substential} we obtain $$
    \mu (\mathcal G') < \mu (\mathcal G) - \mu (\mathcal T)/3 + \eps \mu (\mathcal T) + \eps \mu (\mathcal G) \leq \frac{1}{3}(2 + 5\eps) \mu (\mathcal G).
    $$
Choosing $\eps < 1/5$ completes the proof.
\end{proof}

\subsection{Decreasing sequence of substantial subgraphs}\label{sec:substantial}

A fundamental step in our proof of Theorem~\ref{measure_main} is to show that existence of a measurable one-ended spanning tree is equivalent to the existence  of a nested sequence of \emph{substantial} measurable subgraphs $(\mathcal G_n)_{n\in \mathbb N}$ of $\mathcal{G}$ such that $\mu(\mathcal{G}_n)\to0$. This is precisely Proposition~\ref{basic}, which this section is devoted to prove.

In the next lemma we show that, in an analogous fashion to the situation described in example \ref{ex:substantial}, whenever $\mathcal{G}' \subseteq \mathcal{G}$ is substantial, we can attach the isolated vertices to $O( \mathcal{G}')$ as part of a forest of finite trees. 

\begin{lemma}\label{lem:technical}
    Let $\mathcal{G}$ be a locally finite, one-ended, hyperfinite, mcp graph on a finite measure space $(X,\nu)$ and $\mathcal{G}' \subseteq \mathcal{G}$ a substantial measurable subgraph. Then, for any fber $\mathcal{F} \subseteq \mathcal{R}|_{\mathcal G[I(\mathcal G')]}$ and $\delta > 0$ there exists an acyclic Borel subgraph $\mathcal{T} \subseteq \mathcal{G}$ with finite connected components such that 
    \begin{itemize}
        \item[i)] $|[x]_{\mathcal{T}} \cap O(\mathcal{G}')| = 1$ for $\nu$-a.e.\ $x\in I(\mathcal{G}')$, and
        \item[ii)] $[x]_\mathcal{F} \subseteq [x]_{\mathcal{T}}$ for at least $(1 - \delta) \nu (I(\mathcal{G}'))$-many $x\in X$.
    \end{itemize}
\end{lemma}

\begin{proof}
    By hyperfiniteness of $\mathcal{G}$ there exists a fber $\mathcal E$ such that $\mathcal{F} \subset \mathcal{E} \subset \mathcal{R}|_{\mathcal G[I(\mathcal G')]}$, has $\mathcal G$-connected equivalence classes, each of whose intersection with $I(\mathcal G')$ is also connected, and  \begin{equation*}
    \nu \left(x \in I(\mathcal G') : \exists y \in O(\mathcal{G}')\ \text{such that}\ (y,x) \in \mathcal{E}\right) \ge (1-\delta) \nu (I(\mathcal{G}')).
    \end{equation*}
    We construct $\mathcal{T}$ as the union of three distinct trees $\mathcal T_1, \mathcal T_i$, and $\mathcal T_\infty$. In the first step we construct $\mathcal{T}_1 \subset \mathcal{G} \cap \mathcal E$. Each $\mathcal E$-equivalence class $K$ such that $K \cap O(\mathcal G') \neq \emptyset$ contains a connected component of $\mathcal T_1$. This connected component is any spanning tree of $K \cap I(\mathcal G')$, which is connected by hypothesis, attached to a unique vertex in $K\cap O (\mathcal{G}')$ by a single vertex. The graph $\mathcal T_1$ already satisfies the conclusion ii). In the next step of the proof we extend its domain so that i) is met. 
    
    Let $A\subset I$ be the Borel subset consisting of those $x\in I$ such that $[x]_E \cap O (\mathcal{G}') = \emptyset$. Let $A = A_f \sqcup A_\infty$ be a Borel partition into unions of connected components of $\mathcal{G}[A]$ such that $\mathcal{G}[A_f]$ has finite connected components and $\mathcal{G}[A_\infty]$ has infinite connected components. 
    
    We let $\mathcal{T}_f \subset \mathcal G$ be defined as follows. For each connected component $K$ of $\mathcal G [A_f]$, the forest $\mathcal T_f$ contains a connected component which is given by any spanning subtree of $\mathcal G [K]$, plus one single edge attaching said subtree to $X \backslash A$. 
    
    To conclude the proof, our goal is to partition $A_\infty$ into finite $\mathcal G$-connected subgraphs adjacent to $X\backslash A$.  To this end, fix a Borel linear order\footnote{This may be substituted by a factor of iid argument using a tie-breaking induced by Uniform$[0,1]$ labels which are all distinct almost surely.} $\leq$ on $X$ and for each $x \in \partial_v^{\mathrm{out}} (A_\infty)$ define 
    \begin{equation}\label{eq:voronoi}
        V_{x} := \{ y \in A_\infty : \forall z \in \partial_v^{\mathrm{out}} (A_\infty), d_{\mathcal{G}} (x,y) \leq d _{\mathcal{G}} (z,y)\ \text{and in case of equality}\ x \leq z\}.
    \end{equation}
    The Borel family of sets $(V_{x})_{x\in \partial_v^{\mathrm{out}} (A_\infty)}$ can be understood as a Voronoi tessellation of $A_\infty$ with centers in the boundary $\partial_v^{\mathrm{out}} (A_\infty)$. If $y \in A_\infty$ we let $c(y) \in B$ be the unique element of $B$ such that $y\in V_{c(y)}$. If the $V_{x}$ were finite for a.e.\ $x\in X$, then the proof would be over. This would be the case, for instance, in the pmp setting. However, this is not guaranteed in the mcp setting. To conclude the proof we need the following claim. It shows that the $V_x$ can be locally edited so that only an arbitrarily small proportion of vertices is contained in an infinite $V_x$. 

    \begin{claim}\label{claim:partition}
        Suppose that $(V_x)_{x\in \partial_v^{\mathrm{out}} (A_\infty)}$ is a partition of $A_\infty \cup \partial_v^{\mathrm{out}} (A_\infty)$ such that, for each $x\in \partial_v^{\mathrm{out}} (A_\infty)$, the induced subgraph $\mathcal G[V_x]$ is connected and $x\in V_x$. Then, for any $\eta > 0$ there exists a partition $(W_x)_{x\in B}$ of $A_\infty \cup \partial_v^{\mathrm{out}} (A_\infty)$ such that, for each $x\in \partial_v^{\mathrm{out}} (A_\infty)$, the induced subgraph $\mathcal G[W_{x}]$ is connected, satisfies $x\in W_x$, and 
        $$
            \nu\left(y \in A_\infty :  |W_{c(y)}|=\infty\right)\leq \eta \nu (A_\infty).
        $$
    \end{claim}

    \begin{proof}[Proof of claim]
     For each $x\in \partial_v^{\mathrm{out}} (A_\infty)$ and $k\in \mathbb N$, let $$
     V_x^{\leq k}:= \{y\in V_x : d_{\mathcal G} (y,x) \leq k\},
     $$
     and denote its complement $V_x \backslash V_{x}^{\leq k}$ by $V_x^{> k}$. Note that by local-finiteness $\mathcal G [V_x^{>k}]$ has at most finitely many connected components. For simplicity of the exposition we will assume that each $\mathcal G[V_x^{>k}]$ is actually connected. The following proof can be easily adapted to the full generality treating each connected component in the same fashion. 
     
     Fix $k\in \mathbb N$ large enough such that $$
    \nu \left(\bigcup_{x\in \partial_v^{\mathrm{out}} (A_\infty)} V_x^{\leq k}\right) \geq (1 - \eta/3) \nu (A_\infty).
    $$
    Let us call the union in the above expression $V^{\leq k}$ and its complement $V^{>k} := A_\infty \backslash V^{\leq k}$.
    
    Let $B \subseteq \partial_v^{\mathrm{out}} (A_\infty)$ be the Borel subset of those $x\in \partial_v^{\mathrm{out}} (A_\infty)$ for which $|\partial_v^{\mathrm{out}} (V_x^{>k}) \cap V^{\leq k} | = \infty$. By one-endedness and local finiteness, if $x \in B$ then $\partial_v^{\mathrm{out}} (V_x^{>k}) \cap V^{\leq k}$ intersects infinitely many other $V_y$ non-trivially. It follows that there exists $C \subseteq \partial_v^{\mathrm{out}} (A_\infty)$ such that every $x \in B$ admits a $y \in C$ such that $\partial_v^{\mathrm{out}} (V_x^{>k}) \cap V_y^{\leq k} \neq \emptyset$ and $C$ has small measure:
    $$
    \nu \left(\bigcup_{x\in C} V_x\right) \leq \frac{\eta}{3} \nu (A_\infty).
    $$
    Using the Lusin--Novikov Theorem, construct a measurable choice $y(x) \in C$ for each $x\in B$ such that $\partial_v^{\mathrm{out}} (V_x^{>k}) \cap V_{y(x)}^{\leq k} \neq \emptyset$. For each $x\in B$ we let the set $W_x$ to be $V_x^{\leq k}$ and place the remaining part $V_x^{> k}$ into $W_{y(x)}$. 
    
    % and such that every $x\in C$ has an $\mathcal H$-neighbour in $D$. For each $x \in C$, let $W_x := V_x^{\leq k}$, and for each $x \in D$ let $W_x$ be obtained by attaching the $V_y^{> k}$ of those $y$ such that $s (y) = x$, where $s$ is a Lusin--Novikov uniformization of $\mathcal H \cap (D \times C)$.

    Let $D := \partial_v^{\mathrm{out}} (A_\infty) \backslash B$. We define a Borel graph $\mathcal H\subseteq\mathcal{R}_{\mathcal{G}}$ on $D$ by letting $(y,x) \in \mathcal H$ if and only if $\partial_v^{\mathrm{out}} (V_x^{> k}) \cap V^{> k}_y \neq \emptyset$. By one-endedness of $\mathcal{G}$ the graph $\mathcal{H}$ is aperiodic. Since $\mathcal G$ is hyperfinite, so is $\mathcal H$. Thus, there exists a Borel subgraph $\mathcal K \subset \mathcal H$ with finite connected components and a Borel transversal $D' \subset D$ such that 
    $$
    \nu \left(\bigcup_{x\in D'} V_x\right) \leq \frac{\eta}{3} \nu (A_\infty).
    $$
    For each $x\in D$ we let $z(x) \in D'$ be the unique element of $D'$ which is $\mathcal K$-connected to $x$. For each $x\in D \backslash D'$ we let $W_x$ be $V_x^{\leq k}$ and place the remainder $V_x^{>k}$ in $W_{z(x)}$ where $y \in D'$ the unique element of the transversal in the $\mathcal K$-connected component of $x$. 

    The above process concludes with a partition $(W_x)_{x\in \partial_v^{\mathrm{out}} (A_\infty)}$ that satisfies the desired properties.

    \end{proof}

   To conclude the proof of the lemma set $V_x^{(1)} := V_x$ for each $x\in \partial_v^{\mathrm{out}} (A_\infty)$, and for each $n\in \mathbb N$ given $(V_x^{n})_{x \in \partial_v^{\mathrm{out}} (A_\infty)}$, construct $(V_x^{n+1})_{x \in \partial_v^{\mathrm{out}} (A_\infty)}$ applying the claim with $\eta = 2^{-n}$. By the Borel--Cantelli lemma, for $\nu$-a.e.\ $x\in (V_x^{n})_{x \in \partial_v^{\mathrm{out}} (A_\infty)}$ there exists a least $n_x \in \mathbb N$ such that $V_x^{(N)} = V_x^{(n_x)}$ for every $N \geq n_x$. The partition $(V_x^{(n_x)})_{x \in \partial_v^{\mathrm{out}} (A_\infty)}$ satisfies the desired properties.  
\end{proof}

Under the assumption that substantial subgraphs with arbitrarily small measure exist, we can iteratively apply the above lemma to produce measurable one-ended spanning trees.

\begin{proof}[Proof of Proposition~\ref{basic}]
If $\mathcal{G}$ admits a measurable one-ended spanning tree $\mathcal{T} \subset \mathcal{G}$ the sequence $(\mathcal G_n)$ can be obtained by considering the subgraphs $\mathcal T[L^{>n}]$ where $L^{>n}$ is as in Example \ref{ex:substantial}.

To prove the other direction, we argue by induction. Set $\mathcal G_0: = \mathcal G$ and $\mathcal T_0 $ to be the empty graph. For every $n\in \mathbb N$, let $\mathcal T_n$ be constructed by applying Lemma \ref{lem:technical} to $\mathcal G_{n}$ as a substantial subgraph of $\mathcal G_{n-1} \cup \mathcal T_{n-1}$, setting parameters in the statement of the lemma $\mathcal F = \mathcal E_{n} \cap \mathcal{R}|_{\mathcal G[I(\mathcal G_{n})]}$ and $\delta = 2^{-n}$.

We claim that $\mathcal T = \bigcup_{n\in \mathbb N} \mathcal T_n$ is a measurable one-ended spanning tree for $\mathcal G$. By construction $\mathcal T \subset \mathcal G$ and $\mathcal T$ is acyclic as $(\mathcal T_{n})_{n\in \mathbb N}$ is an increasing union of trees. By the Borel--Cantelli lemma and the fact that $\mu (\mathcal G_n) \to 0$, for $\nu$-a.e.\ $x\in X$ we have that for every $m\in \mathbb N$ the inclusion $[x]_{\mathcal E_m} \subset [x]_{\mathcal T_n}$ holds for all $n$ sufficiently large. Thus, $\mathcal T$-components coincide with $\mathcal G$-components a.e. 

It is only left to show that $\mathcal T$ is one-ended. To do this, let $e \in \mathcal T$ and let $n$ be such that $e \in \mathcal T_n$. Only one connected component of $\mathcal T_n \backslash \{e\}$ intersects $O(\mathcal G_n)$, as each $\mathcal T_n$-connected component intersects $O(\mathcal G_n)$ at exactly one vertex. No edge can be added to any vertex in the other connected component of $\mathcal T_n \backslash \{e\}$ throughout the construction. Therefore $\mathcal T \backslash \{e\}$ has  at most one infinite connected component for every $e$, and so $\mathcal T$ is one-ended.
\end{proof}

\subsection{Deleting edges outside of a measurable tree}\label{sec:out of tree}

In this section we show that we can find substantial subgraphs with edge mass as close to a measurable spanning tree as desired. This is precisely Proposition \ref{tree measure}.

The following proposition ensures conditions under which edges of a one-ended Borel graph can be deleted whilst preserving one-endedness.

\begin{prop}\label{prop:stays.one.ended}
    Let $\mathcal{G}$ be a locally finite one-ended Borel graph and $\mathcal{F} \subseteq \mathcal{R}_{\mathcal{G}}$ a fber with $\mathcal{G}$-connected classes. Let $Q \subseteq \mathcal{F}$ be a Borel subset such that $\mathcal{F}$-classes are still $\mathcal{G} \backslash Q$-connected. Then $\mathcal{G} \backslash Q$ is one-ended. 
\end{prop}

\begin{proof}
Consider an arbitrary  connected component $G$ of $\mathcal{G}$ and let $H := G \backslash Q$. Our goal is to show that for any finite $C \subset E(H)$ the graph $H \backslash C$ has a unique infinite connected component. The $\mathcal{F}$-saturation $K$ of the set of end-vertices of edges in $C$ is finite. Therefore, by hypothesis $G \backslash K$ has a unique infinite connected component $U$. By local finiteness we have, moreover, that $|V(G) \backslash U| < \infty$.

Any two vertices $u,v \in U$ are connected by a path $P$ in $G\backslash K$. If $e \in E(P) \cap Q$, then $e \in \mathcal{F}$, so since $\mathcal{F}$-classes are $\mathcal{G}\backslash Q$-connected, there exists a path $P_e$ between the end-vertices of $e$, with no edge from $Q$, and contained in the $\mathcal{F}$-saturation of $e$. Therefore, substituting such edges $e$ by the corresponding paths $P_e$ we obtain a path $P'$ in $G \backslash (Q \cup K)$ between $u$ and $v$. It follows that $U$ is connected in $H \backslash C$. Since $U$ is cofinite $H\backslash C$ can only have one infinite connected component.
\end{proof}

Applying the above proposition in presence of a measurable spanning tree yields the main result of the section.

\begin{proof}[Proof of Proposition \ref{tree measure}]
    Let $(E_n)_{n\in \mathbb N}$ be a witness to hyperfiniteness of $\mathcal G$ such that $E_n$-equivalence classes are $\mathcal T$-connected. For each $n \in \mathbb N$ we define a Borel graph $\mathcal G_n'$ by deleting from $\mathcal G$ all edges in $E_n$ which do not belong to $\mathcal T$. Precisely, these graphs are defined by $$
    \mathcal G_n' := \mathcal T \cup (\mathcal G \cap (R_{\mathcal G} \backslash E_n).
    $$
    The $\mathcal G_n'$ are one-ended by Proposition \ref{prop:stays.one.ended}, are such that $\mathcal R_{\mathcal G_n'} = \mathcal R_{\mathcal G}$ and $\mu (\mathcal G_n) \to \mu (\mathcal T)$.
\end{proof}

\subsection{Deleting edges outside of the core}\label{finite}

Towards proving Theorem \ref{measure_main}, we can assume that most of the mass of $\mathcal G$ is contained in $\mathcal T$ after Proposition \ref{tree measure}. Hence, in order to make progress in deleting edges, we must focus our attention on deleting edges of $\mathcal T$ in a way that produces substantial subgraphs. Our analysis of this setting revolves around the notion of core, introduced in Definition \ref{def:core}. The following proposition serves as a motivation to the study of the core.

\begin{prop}\label{fundamental}
    Let $\mathcal{T}$ be a measurable spanning tree of a one-ended locally finite mcp graph $\mathcal{G}$ on a standard probability space $(X,\nu)$. Then $\mathcal{T}$ is one-ended if and only if $\mathrm{Core}_\mathcal{T}(\mathcal{G})=\emptyset$ a.e.
\end{prop}

\begin{proof}
    Removing any edge $e \in \mathcal T$ from its $\mathcal T$-connected component results in two connected components, say $A$ and $B$. The number of fundamental cycles $e$ is contained in is, precisely, the number of $\mathcal G$-edges between these two connected components. Thus, if $\mathcal T$ is one-ended, then one of these connected components is finite, and so the number of $\mathcal G$-edges between $A$ and $B$ is finite by local finiteness of $\mathcal G$. We deduce that if $\mathcal T$ is one-ended then $\mathrm{Core}_\mathcal{T}(\mathcal{G})=\emptyset$ a.e.

    To prove the converse, observe that if $\mathcal T$ has a component with at least two ends, then there exists $e\in \mathcal T$ in such component for which $A$ and $B$ above are infinite. Since $\mathcal G$ is one-ended, there must be infinitely many $\mathcal G$-edges between $A$ and $B$. It follows that such an $e$ does not belong to the core, concluding the proof.
\end{proof}

The study of the core is related to the study of generating sets of cycle spaces.

\begin{definition}
    A \textbf{Borel generating set} for the cycle space a Borel graph $\mathcal{G}$ is a Borel family of cycles $\mathcal B \subset [\mathcal{G}]^{< \infty}$ whose indicator functions generate the cycle space $\mathcal C( \mathcal{G})$.
\end{definition}

\begin{example}
    If $\mathcal{T} \subset \mathcal{G}$ is a Borel spanning tree, then the Borel family of fundamental cycles is a Borel generating set. In fact it is moreover a basis of the cycle space. This basis is \textbf{feasible}, in the sense of \cite{Timar_oneended}, i.e., every edge is contained in only finitely many cycles of $\mathcal B$, when both $\mathcal T$ and $\mathcal G$ are one-ended. 
\end{example}

\begin{definition}
Let $\mathcal{G}$ be a Borel graph with Borel generating set for its cycle space $\mathcal B$. For each $e\in \mathcal{G}$ we let $$
e^\ast := \{C \in \mathcal B : e \in E(C)\}.
$$

For a Borel subset $U\subset X$, we define the \textbf{kernel of $U$ (with respect to a Borel generating set $\mathcal B$)} as $$
\ker (U) := \{e \in \mathcal{G} : e^\ast \subset \mathcal{G}[U]\}.
$$ 
\end{definition}

If $\mathcal G$ is a one-ended Borel graph and $H \subset \mathcal G$ is a finite subgraph, we denote by $\ext (H) \subset V(H)$ the \textbf{exterior of $H$}: the set of vertices in $V(H)$ belonging to an $H$-connected component for which there exists an infinite ray $R$ in $\mathcal G$ intersecting $H$ only at its starting vertex.

\begin{lemma}\label{lem:exterior.conn}
    Suppose that $\mathcal{G}$ is one-ended and let $F\subset [
    X]^{<\infty}$ be such that the induced subgraph $\mathcal{G} [F]$ is connected. Then $\ext(\mathcal G[F] \backslash \ker (F))$ is connected.
\end{lemma}

\begin{proof}
    Let $P \subset \mathcal{G}[F]$ be a path with end-vertices in $\ext(\mathcal G[F] \backslash \ker (F))$ and containing edges of $\ker (F)$. We show that there exists a path in $\mathcal G[F]$ disjoint from $\ker (F)$ with the same end-vertices as $P$. 
    Since the end-vertices of $P$ are in the exterior there exists a path $Q$ in $\mathcal{G} \backslash \ker (F)$ with the same end-vertices as $P$ and disjoint from $P$. Since $\mathcal B$ is generating, there exist $B_1, \dots, B_n \in \mathcal B$ such that $$
    Q+P = B_1 +\ldots + B_n,
    $$
    in the edge space. We conclude the proof by induction. If $n=1$, we have that $Q + P \in \mathcal B$. Since there are edges of the kernel in $Q+P$, we have that $Q+P \subset F$, so $Q \subset F \backslash \ker (F)$.

    Let us now fix $n$ and suppose that for every $k < n$, if $P + Q = B_1 +\dots + B_k$, then there exists $Q'$ with the same end-vertices as $P$ and disjoint from $\ker (F)$. If $P + Q = B_1 +\dots + B_n$ is contained in $F$ we are done. Otherwise, say, $B_n \subsetneq F$. This implies that $B_n \cap \ker (F) = \emptyset$ by definition of the kernel. We thus have that $$
    P + Q + B_n = B_1 + \ldots+ B_{n-1},
    $$
    is such that $Q' + B_n$ satisfies the induction hypotheses. The conclusion now follows.
\end{proof}

\begin{remark}
    An alternative proof for the above statement may be obtained using the boundary connectivity results of Tim\'ar \cite{TimarBC}.
\end{remark}

As a conclusion of the above lemma we deduce that we can remove arbitrarily large proportions of edges in the complement of $\mathrm{Core}_\mathcal{T}(\mathcal{G})$ resulting in a substantial subgraph as stated in Proposition~\ref{boundary connectivity}. 

\begin{proof}[Proof of Proposition~\ref{boundary connectivity}]
By hyperfiniteness of $\mathcal{G}$ there exists an fber $\mathcal E \subseteq \mathcal R_{\mathcal G}$ such that $(1- \varepsilon)$ many of the edges of $\mathcal{T}\setminus \mathrm{Core}_\mathcal{T} (\mathcal{G})$ belong to the kernel of their $\mathcal E$-classes. We can remove these edges preserving connectivity by Lemma \ref{lem:exterior.conn} and one-endedness by Proposition \ref{prop:stays.one.ended}.
\end{proof}

\begin{remark}
    The above argument can be adapted to show that any one-ended, hyperfinite Borel graph that admits a Borel feasible generating set also admits a Borel one-ended spanning forest.  In particular, any one-ended, hyperfinite, planar Borel graph admits a Borel one-ended spanning forest.  Under these assumptions, this refines a result of Conley, Gaboriau, Marks and Tucker-Drob \cite[Theorem 10]{cgmtd}, who proved the same thing in the measurable setting but without the one-ended and hyperfinite assumptions. Note that in the Borel setting the one-ended assumption is necessary, even for hyperfinite graphs, as there are acyclic (hence planar) and hyperfinite counterexamples to the Borel Brooks' theorem \cite{conley2020borel}.
\end{remark}

After Proposition \ref{boundary connectivity} we can assume that most of the mass of $\mathcal T$ lies in the core. The following lemma ensures that the connected components of the core, which bear a small mass by assumption, can be contracted to points in a Borel fashion preserving local finiteness and one-endedness.

\begin{lemma}\label{nocore}
    Let $\mathcal{G}$ be a one-ended hyperfinite Borel graph with measurable spanning tree $\mathcal{T}$. Then {$\mathcal T \backslash \mathrm{Core}_{\mathcal{T}}(\mathcal{G})$} is either equal to $\mathcal{T}$ or it has finite connected components a.e.
\end{lemma}
\begin{proof}
    Let $G$ be a connected component of $\mathcal G$ such that $T := G\cap \mathcal T$ is a spanning subtree of $G$. Suppose that there exists an infinite connected component in $G\backslash \mathrm{Core}_{T}(G)$. Since $\mathcal G$ is locally finite there exists a ray $R \subset G \backslash \mathrm{Core}_{T}(G)$ by K\H{o}nig's lemma. 
    
    By Proposition \ref{fundamental}, if $T$ is one-ended, then $\mathrm{Core}_{T}(G) = \emptyset$ and the proof is over. Otherwise, if $T$ has more than one end, since $R$ represents one of the ends of $T$, there exists $e \in E(R)$ such that $\mathcal T \backslash \{e\}$ consists of two infinite connected components. By one-endedness of $G$, there must exist infinitely many edges between the two connected components. But each of such edges induces a fundamental cycle containing $e$. Therefore $e \in \mathrm{Core}_{T}(G)$, a contradiction.
\end{proof}

\begin{remark}[Empty core]\label{nocore}
Iteratively applying Proposition \ref{boundary connectivity}, we can assume that we've produced a substantial subgraph $\mathcal{G}'\subset \mathcal{G}$ with measurable spanning tree $\mathcal T'$ such that $\mu(\mathcal T' \backslash \mathrm{Core}_{\mathcal{T}'}(\mathcal{G}'))/\mu(\mathcal{G}')$ is as small as desired.  If not, then we would be done by Proposition~\ref{basic}.  Consequently, up to contracting a set of edges in $\mathcal{T}'$ of arbitrarily small measure, we can assume that a.e.\ edge in $\mathcal{T}'$ is in infinitely many fundamental cycles.  In particular, we may assume that $\mathcal{T}'$ has no leaves and $\mathrm{Core}_{\mathcal{T}'}(\mathcal{G})=\mathcal T'$.
\end{remark}

\subsection{Deleting edges in the core}\label{infinite}

This is the last section in the proof of Theorem \ref{measure_main} and has the goal of proving Proposition \ref{sparse_substential}.  From Remark \ref{nocore}, we may assume that $\mathcal{G}$ has a measurable spanning tree $\mathcal{T}$ with $\mathrm{Core}_{\mathcal{T}}(\mathcal{G})=\mathcal T$.  Notice that if we did not have this assumption, then $\mathcal{G}$ could have a set of leaves of large measure whose removal might result in a non-substantial subgraph.  Remark \ref{nocore} allows us to assume that $\mathcal{T}$ has no leaves, which we will later use to argue that $\mathcal{T}$ is $3$-edge-connected.

Recall that $\mathcal{G}$ is a locally finite, one-ended, hyperfinite, mcp graph on a standard probability space $(X,\nu)$. As we mentioned in Subsection~\ref{sec:prelim_measured}, $\mathcal{G}$ is naturally equipped with a relative weight function $\w$ given by the Radon--Nikodym cocycle of its connectedness relation with respect to $\nu$. By \cite[Lemma 3.21]{jkl}, since $\mathcal{G}$ is hyperfinite it admits a measurable spanning tree $\mathcal{T}$ that admits a Borel selection of one or two of its ends in each component. In the latter case one can derive Theorem~\ref{measure_main} using similar argument to that in the proof of Lemma~\ref{lem:technical}.

\begin{prop}\label{2 ends}
    Let $\mathcal G$ be a one-ended hyperfinite locally finite mcp graph and $\mathcal T$ a measurable spanning tree. If $\mathcal T$ admits a Borel selection of two of its ends in a.e.\ component, then $\mathcal{G}$ admits a measurable one-ended spanning tree.
\end{prop}

\begin{proof}
    If $\mathcal T$ admits a Borel selection of two ends in a.e.\ component, there exists a Borel graph $\mathcal L \subset \mathcal T$ such that a.e.\ $\mathcal T$-connected component contains exactly one $\mathcal L$-connected component and this connected component is a bi-infinite line. 
    
    The proof follows as in the analysis of $A_\infty$ in Lemma \ref{lem:technical}. Let $Y$ denote the Borel subset of vertices of the bi-infinite lines in $\mathcal L$ and $Z := X \backslash Y$. A Voronoi tessellation of $Z$ defined as in Equation (\ref{eq:voronoi}) yields a Borel partition $(V_x)_{x\in Y}$ of $X$ such that for each $x\in Y$, the induced subgraph $\mathcal G [V_x]$ is connected and $x \in V_x$. Applying Claim \ref{claim:partition} iteratively with suitably chosen parameters as in the end of the proof of Lemma \ref{lem:technical} we obtain a partition $(W_x)_{x\in Y}$ of $X$ such that $|W_x|<\infty$ for a.e.\ $x\in X$, the induced graph $\mathcal G[W_x]$ is connected, and $x\in X$. 
    
    Letting $\mathcal T_x$ denote a spanning tree of $\mathcal G[W_x]$ for each $x\in Y$, it follows that the graph $\mathcal T' := \mathcal L \cup \left(\bigcup \mathcal T_x\right)$ is a measurable two-ended spanning tree of $\mathcal G$. Then $\mathcal G$ admits a measurable one-ended spanning tree by \cite[Lemma 2.9]{cgmtd}.    
\end{proof}

Hence for the purposes of proving Theorem~\ref{measure_main} we may assume that the spanning tree $\mathcal{T}$ of $\mathcal{G}$ admits a Borel selection of exactly one end in each of its components. For an edge $(x,y)$ in such a tree $\mathcal{T}$ let $\max_\mathcal{T}\{x,y\}$ be the point that is further from the selected end. We now extend the relative weight function $\w$, as in Section~\ref{sec:prelim_measured}, to the edges of $\mathcal{T}$ by setting $\w^y(x,y):=\w^y(\max_\mathcal{T}\{x,y\})$.

By Lemma \ref{2 ends} and \cite[Theorem E]{miller.2end} we can assume that $\mathcal{T}$ has infinitely many ends. In particular, there are vertices with $\mathcal{T}$-degree at least $3$.  Call maximal $\mathcal{T}$-paths all of whose vertices have $\mathcal{T} $-degree $2$ a \textbf{segment}.  Call an edge $e\in \mathcal{G}\setminus \mathcal{T}$ \textbf{redundant} if it goes between two vertices on the same segment.  Note that segments may be of infinite length (in which case it also has to be $\w$-finite) but cannot be bi-infinite, as otherwise they would be a $2$-ended spanning tree of their component contradicting to the fact that $\mathcal{T}$ has infinitely many ends.

\begin{lemma}
    Removing the redundant edges from $\mathcal{G}$ produces a substantial Borel subgraph.
\end{lemma}
\begin{proof}
    The end-vertices of a redundant edge are connected by a segment in $\mathcal{T}$. Since no edge in $\mathcal{T}$ is redundant, removing redundant edges preserve connectivity. It thus suffices to show that the Borel graph $\mathcal{G}'$ obtained by removing the redundant edges from $\mathcal{G}$ is one-ended. 
    
    To this end consider a connected component $G$ of $\mathcal{G}$, containing the corresponding connected components $G'$ of $\mathcal{G}'$ and $T$ of $\mathcal{T}'$. Let $K$ be a finite subset of edges of $G'$ and $u,v\in V(G)$ vertices in infinite connected components of $G' \backslash K$. In particular both $u$ and $v$ lie in infinite connected components of $G \backslash K$, so by one-endedness there exists a path $P \subset G \backslash K$ from $u$ to $v$. Any redundant edge in $P$ may be substituted by a path in some segment of $\mathcal{T}$. Therefore, we may assume without loss of generality that $P \subset G' \backslash K$. It follows that $G' \backslash K$ has a unique infinite connected component. Since $K$ was arbitrary, we deduce that $G'$ is one-ended, concluding the proof. 
    \end{proof}

Observe that an edge in $\mathcal  G\setminus \mathcal  T$ is not redundant if and only if the unique fundamental cycle it induces contains a vertex with $\mathcal{T}$-degree at least three.  Removing redundant edges, we may assume that every $\mathcal{G}\setminus \mathcal{T}$-edge induces such a fundamental cycle.

Consider any path that consists of vertices with $\mathcal{T}$-degree $2$ that are not adjacent to any of the edges in $\mathcal{G}\setminus \mathcal{T}$.
Such paths necessarily have finite length, or we would have contradicted the fact that $\mathcal G$ is one-ended;
indeed if there was such an infinite path it would constitute an isolated ray in $\mathcal{G}$, hence contradicting one-endedness assumption on $\mathcal{G}$. We can safely contract this path into an edge whose weight is set to be equal to the total $\w$-weight of the edges in the path. Equivalently, one can collapse the vertices in this path with the vertex in its outer boundary point that is furthest the from the selected end in $\mathcal{T}$ yielding a single vertex whose weight is set to be equal to the total $\w$-weight of the path and the $\w$-weight of the end point.
In particular, after doing this contraction we can assume that all vertices in non-trivial $\mathcal{G}$-components have $\mathcal{T}$-degree at least $3$ or else have $\mathcal{T}$-degree $2$ and at least one edge in $\mathcal{G}\setminus \mathcal{T}$. 

Recapping, we can assume that every edge in $\mathcal{T}$ is in infinitely many fundamental cycles and that every vertex in a non-trivial $\mathcal{G}$ component either has $\mathcal{T}$-degree at least $3$ or has $\mathcal{T}$-degree $2$ and at least one edge that induces a fundamental cycle meeting a vertex of $\mathcal{T}$-degree at least $3$.  The latter follows from our discussion in the preceding paragraph and the fact that $\mathcal{T}$ has no degree $1$ vertices since we mar assume that $\mathrm{Core}_{\mathcal{T}}(\mathcal{G})=\mathcal T$ by Remark \ref{nocore}.

\begin{lemma}\label{3con}
    Every non-trivial $\mathcal{G}$ component is $3$-edge-connected.
\end{lemma}
\begin{proof}
    Towards a contradiction suppose that there are two edges $e,f$ whose removal cuts a $\mathcal{G}$-connected component $G$ into multiple connected components. Since $G$ is one-ended, exactly one of the connected components of $G \backslash \{e,f\}$ is infinite.  Since $\mathcal{T}$ has degree at least 2, every vertex of $G$ is the first vertex of three edge-disjoint rays. Therefore, any connected component of $G \backslash \{e,f\}$ containing a vertex of $\mathcal{T}$-degree at least 3 must be infinite. Since there is exactly one infinite connected component of $G \backslash \{e,f\}$, it follows that all vertices with $\mathcal{T}$-degree at least 3 belong to the same $G \backslash \{e,f\}$ connected component. 
    
    In particular, we may assume that both $e$ and $f$ are on a $\mathcal{T}$-degree two path and isolate a finite subpath.  But every vertex on this subpath has an edge whose fundamental cycle meets a $\mathcal{T}$-degree $3$ vertex, a contradiction.
\end{proof}

A graph $\mathcal{G}$ is \textbf{locally} $3$-\textbf{edge connected} if for every $\{x,y\}\in E(G),$ there are $3$ edge-disjoint paths from $x$ to $y.$ 
\begin{lemma}\label{Menger}
    A graph is locally $3$-edge-connected if and only if each of its connected components is $3$-edge-connected.
\end{lemma}

\begin{proof}
    The if direction is Menger's theorem.  For the only if direction, look at a potential edge cut of size $2$ and consider the third path between one of the cut edges.
\end{proof}

We are finally ready for our last edge deletion procedure.

\begin{proof}[Proof of Proposition~\ref{sparse_substential}]
    Let $(\mathcal{G}_n)_{n\in \mathbb N}$ be a witness to hyperfiniteness of $\mathcal{G}$ and $\mathcal{T}$ be the spanning subtree of $\mathcal{G}$. In light of Proposition~\ref{2 ends} we may assume that $\mathcal{T}$ admits a Borel selection of a unique end in a.e.~$\mathcal{T}$-component (in case of selection of 2 ends, Proposition \ref{2 ends} yields a stronger conclusion). As above, extend the relative weight function $\w$ to the edges of $\mathcal{T}$ and set it to be equal to $0$ for the edges in $\mathcal{G}\setminus\mathcal{T}$.
    
    Fix $\eps>0$. By Lemma~\ref{3con} $\mathcal{G}$ is $3$-edge-connected and by Lemma~\ref{Menger} it is locally $3$-edge-connected. Thus there is $N\in \mathbb N$ sufficiently large such that
    $$
        \mu\left((y,x) \in \mathcal{G} : \exists\ 3\ \text{edge-disjoint paths in $\mathcal{G}_N$ between}\ x\ \text{and}\ y\right) > (1 - \varepsilon) \mu (\mathcal{G}).
    $$
    In each $\mathcal{G}_N$-connected component contract edges whose end-vertices are not contained in $3$-edge disjoint paths. This results in a Borel family of finite locally $3$-edge-connected graphs. By Lemma~\ref{Menger} these graphs are $3$-edge-connected and account for $1-\varepsilon$ fraction of the total edge measure in $\mathcal{G}.$ 
    
    Duplicate each edge (giving each new edge the same $\w$-weight as it used to have), producing a $6$-edge-connected multigraph.  By the Nash-Williams theorem \cite{nash1961edge} there are three edge-disjoint spanning trees of this multigraph.  We claim that one of them corresponds to a spanning tree of the unduplicated graph that has at most $2/3$ of the total $\w$-weight.

    To see this, note that each edge in the original graph appears at most twice in the union of edges from the three spanning trees.  In particular, the total $\w$-weight of the spanning trees is at most twice of the original weight of the graph, and hence one of the trees has at most $2/3$ of the total $\w$-weight.

    Finally, delete edges in each $\mathcal{G}_n$ component that are not in the spanning tree we have just produced.  Note that the resulting subgraph is still one-ended by Proposition \ref{prop:stays.one.ended}.
\end{proof}

%%%%%%%%%%%%%%%%%%%%%%%
\section{Implications for percolation theory}\label{sec:perc}
%%%%%%%%%%%%%%%%%%%%%%%
We briefly recall the background for percolation theory and refer the reader to \cite{LyonsBook,PetePGG} for further details. Percolation theory studies random subgraphs of a given countable graph $G$. To be precise, a bond \textbf{percolation process} $\mathbf{P}$ on a graph $G=(V,E)$ is a probability measure on $2^E$. We say that percolation is $\Gamma$-invariant, for a subgroup $\Gamma\subseteq \Aut(G)$, if its law is invariant under the action of $\Gamma$ on $G$. Percolation theory on quasi-transitive graphs dates back to the 1990s \cite{BSbeyond,Haggstrom99,BLPS99inv} and is largely divided into two further settings: the unimodular and nonunimodular. One way to define this notion for a locally finite connected graph is to call such a graph $G$ \textbf{unimodular} if its automorphism group $\mathrm{Aut}(G)$ is unimodular (meaning that the left-invariant Haar measure is also right-invariant); otherwise, $G$ is \textbf{nonunimodular}. We note that for most percolation results in the nonunimodular setting, it is enough to assume the existence of a closed nonunimodular subgroup $\Gamma\subseteq \Aut(G)$; hence, such results are often of interest even in the unimodular setting, e.g.\ \cite{Hutchcroft20}.

The assumption of unimodularity is in many ways parallel to the assumption of probability measure preservation in measured combinatorics and group theory. In fact, the connection between these fields goes beyond a similarity in themes and techniques: results in one often enable advances in the other \cite{Gaboriau05,GL09,CTT22,wamen,HeavyClusterRepulsion}. The main bridge between the two theories is the cluster graphing construction introduced by Gaboriau \cite[Section 2]{Gaboriau05}, see \cite[Section 5]{CTT22} for the presentation tailored to the nonunimodular/mcp settings. 

In the mcp setting, the relative weight function on a Borel graph is induced by the Radon–Nikodym cocycle of its connectedness relation with respect to the underlying measure. For a quasi-transitive graph, it is induced instead by the Haar measure defined on the group of automorphisms. Let $\Gamma$ be a closed subgroup of $\Aut(G)$ that acts quasi-transitively on $G$. The relative weight function (i.e.\ a cocycle) $\w_\Gamma:V^2\to\R^+$ is given by the modular function $\w^y_{\Gamma}(x)=m(\Gamma_x)/m(\Gamma_y)$, where $\Gamma_v := \{\gamma\in\Gamma \mid \gamma v=v\}$ is the stabilizer of $v\in V$ and $m$ is the Haar measure on $\Gamma$. It also has a more combinatorial representation \begin{equation}\label{def:haarweights} 
\w^y_{\Gamma}(x):=\frac{m(\Gamma_x)}{m(\Gamma_y)}=\frac{\abs{\Gamma_x y}}{\abs{\Gamma_y x}}.
\end{equation}
Note that $\Gamma$ is unimodular if and only if $\w_\Gamma^x(y)\equiv 1$ for all $x, y \in V$ from the same $\Gamma$-orbit. 

The notions of amenability and hyperfiniteness are of classical interest in percolation theory, for formal definitions see \cite{BSbeyond,BLPS99inv,URG,wamen}. The equivalences among them were shown in \cite{BLPS99inv} and are analogous to the corresponding results in measured group theory \cite{cfw,kaimanovich1997amenability}.

We are now ready to present the implication (or a translation) of Theorem~\ref{measure_main} to this setting.
\begin{proof}[Proof of Corollary~\ref{perc_mian}]
%\todo{Greg: I am still working on the proof.}
    The proof follows immediately from the cluster graphing construction; indeed, \cite[Theorem 3.9]{BLPS99inv}, and \cite[Theorem 1.4 and Lemma 4.2]{wamen}, which translate amenability of $\Gamma$ to hyperfiniteness of the corresponding cluster graphing, where Theorem~\ref{measure_main} applies. 
    Given a one-ended spanning tree of the cluster graphing, one can derive the same for $G$ easily via the family of component-wise isomorphisms between the connected components of the cluster graphing and $G$.
    
    Alternatively, one can follow our argument in proof of Theorem~\ref{measure_main} almost verbatim to derive the corollary. Indeed, by \cite[Theorem 5.1]{BLPS99inv} the graph $G$ admits a hyperfinite exhaustion. Repeating the steps in our proofs yields a $\Gamma$-invariant random spanning tree that is one-ended almost surely.

    Since the proof of \cite[Theorem 5.1]{BLPS99inv} actually constructs the hyperfinite exhaustion as a factor of i.i.d.\ and the rest of our proof is deterministic, the resulting one-ended random spanning tree is also a factor of i.i.d.
\end{proof}

%%%%%%%%%%%%%%%%%%%%%%%%%%%%%%%%%%%%%
\section*{Acknowledgments}
%%%%%%%%%%%%%%%%%%%%%%%%%%%%%%%%%%%%%

M.B. is supported by Ben Green's Simons Investigator Grant number 376201.\\
H.J.S. was supported by the ERC Starting Grant “Limits of Structures in Algebra and Combinatorics” No.~805495 and the Dioscuri program initiated by the Max Planck Society, jointly managed by the National Science Centre (Poland), and mutually funded by the Polish Ministry of Science and Higher Education and the German Federal Ministry of Education and Research.\\
G.T. was supported in part by the RTG award grant (DMS-2134107) from the NSF and by the ERC Synergy Grant No. 810115 - DYNASNET. 
\bibliographystyle{amsalpha} 
\bibliography{editable} 
\end{document}